\documentclass{article}
\usepackage{latexsym,amssymb,amsmath,amsfonts,pictex,enumerate,dsfont, cite, amsthm}
\usepackage{scalerel,stackengine}
\usepackage[active]{srcltx}
\usepackage{ gensymb }

\theoremstyle{thmstyleone}
\newtheorem{theorem}{Theorem}
\newtheorem{lemma}{Lemma}

\theoremstyle{thmstyletwo}

\theoremstyle{thmstylethree}

\begin{document}

\title{Boundary smoothness conditions for functions in $R^p(X)$}

\author{Stephen Deterding \thanks{Email: deterding@marshall.edu} \\ Department of Mathematics \& Physics \\ Marshall University, Huntington, WV, USA }

\date{}



\maketitle

\begin{abstract}
    Let $X$ be a compact subset of the complex plane and let $R^p(X)$, $2< p < \infty$, denote the closure of the rational functions with poles off $X$ in the $L^p$ norm. In this paper we consider three conditions that show how the functions in $R^p(X)$ can have a greater degree of smoothness at the boundary of $X$ than might otherwise be expected. We will show that two of the conditions are equivalent and imply the third but the third does not imply the other two.
\end{abstract}

\section{Introduction}

Let $X$ be a compact subset of the complex plane and let $R_0(X)$ denote the set of all rational functions with poles off $X$. In this paper we consider the function space $R^p(X)$, $2 < p < \infty$, the closure of $R_0(X)$ in the $L^p$ norm, and the smoothness of the functions in this space. Many results in this area have been determined for $R(X)$, the closure of $R_0(X)$ in the uniform norm, but for $R^p(X)$ less is known. In particular, Wang proposed three conditions that demonstrate how the functions in $R(X)$ can have a greater degree of smoothness at the boundary of $X$ than might otherwise be expected. \cite{Wang2} To describe these conditions, we first need to state a few definitions.

\bigskip

We say that $\phi(r)$ is an admissible function if it is positive and nondecreasing on $(0, \infty)$ and if the associated function $\psi(r) = \frac{r}{\phi(r)}$ is also positive and nondecreasing on $(0, \infty)$ with $\psi(0^+) =0$. Examples of admissible functions are the power functions $\phi(r) = r^{\alpha}$ where $0<\alpha <1$ and the function $\phi(r) = \frac{r}{\log(1+r)}$. Let $\gamma(X)$ denote the analytic capacity of the set $X$ \cite[pg. 196]{Gamelin} and let $R_x^t f(z)$ be the error at $z$ of the $t$-th degree Taylor polynomial of $f$ about $x$, which is defined by

\begin{equation*}
    R_x^t f(z) = f(z) - \sum_{j=0}^t \dfrac{f^{(j)}(x)}{j!} (z-x)^j.
\end{equation*}

\bigskip \noindent Let $A_n(x)$ denote the annulus $\{z: 2^{-(n+1)} < \vert z-x\vert  < 2^{-n}\}$ and let $B(x,r)$ denote the disk centered at $x$ with radius $r$. Let $m$ denote area measure. A set $E$ is said to have full area density at $x$ if $\dfrac{m(B(x,r) \setminus E)}{m(B(x,r))} \to 0$ as $r \to 0$. 

\bigskip

Let $X$ be a compact subset of the plane, let $x \in X$, let $t$ be a non-negative integer, let $\phi(r)$ be an admissible function, and let $\Vert\cdot\Vert_X$ denote the uniform norm on $X$. Wang's three conditions are the following.

\begin{enumerate}[(A)]
    \item For each $\epsilon >0 $ the set 
    
    \begin{align*}
   & \{y\in X: \vert R_x^t f(y) \vert \leq \epsilon \phi(\vert y-x \vert)\vert y-x \vert^t \Vert f \Vert_X \\ &\text{ for all }  f \in R_0(X)\}
    \end{align*}
    
    \bigskip \noindent has full area density at $x$.
    
    \item There exists a representing measure $\mu$ for $x$ on $R(X)$ such that $\mu(x) = 0$ and 
    
    \begin{equation*} 
    \int \dfrac{d \vert \mu \vert(z)}{\vert z-x \vert^t\phi(\vert z-x \vert)} < \infty.
    \end{equation*}
    
    \item The series 
    
    \begin{equation*}
        \sum_{n=1}^{\infty} \dfrac{2^{n(t+1)}\gamma(A_n(x) \setminus X)}{\phi(2^{-n})} 
    \end{equation*}
    
    \bigskip \noindent converges.
    
\end{enumerate}

\noindent The relationship between these conditions has been studied by Wang and O'Farrell. In various papers \cite{O'Farrell, Wang, Wang2} they showed that (B) and (C) are equivalent and imply (A), but (A) does not imply (B) or (C).

\bigskip

We now consider Wang's conditions applied to the space $R^p(X)$. Before we state what these conditions look like in the context of $R^p(X)$, we introduce $q$-capacity, which is the appropriate capacity for $R^p(X)$. For $1 < q < 2$, the   $q$-capacity of a compact set $X$ in the complex plane is denoted $\Gamma_q(X)$ and is defined by 

\begin{equation*}
\Gamma_q(X) = \inf \int \vert \nabla u\vert ^q dm,
\end{equation*}

\bigskip
\noindent
where the infimum is taken over all infinitely differentiable functions $u$ of compact support with $u \equiv 1$ on $X$.

\bigskip

Let $X$ be a compact subset of the plane, let $x \in X$, let $t$ be a non-negative integer, let $2 < p < \infty$ and $q = \frac{p}{p-1}$, and let $\phi(r)$ be an admissible function. Wang's three conditions restated for $R^p(X)$ are 

\begin{enumerate}[(A$'$)]
    \item For each $\epsilon >0 $ the set 
    
    \begin{align*}
   & \{y\in X: \vert R_x^t f(y)\vert  \leq \epsilon \phi(\vert y-x\vert )\vert y-x\vert ^t \Vert f\Vert _{L^p(X)}\\ &  \text{ for all } f \in R_0(X)\}
    \end{align*}
    
    \bigskip \noindent has full area density at $x$.
    
    \item There exists $g \in L^q(X)$ which represents $x$ on $R^p(X)$ such that 
    
    \begin{equation*}
        \int \dfrac{\vert g(z)\vert ^q dm}{\vert z-x\vert ^{qt}\phi(\vert z-x\vert )^q} < \infty.
    \end{equation*}
    
    \item The series 
    
    \begin{equation*}
        \sum_{n=1}^{\infty} \dfrac{2^{n(t+1)q}\Gamma_q(A_n(x) \setminus X)}{\phi(2^{-n})^q} 
    \end{equation*}
    
    \bigskip \noindent converges.
    
\end{enumerate}

\bigskip \noindent In analogy with $R(X)$, we expect that (B$'$) and (C$'$) are equivalent, and imply (A$'$), but (A$'$) does not imply (B$'$) or (C$'$). What is known about the relationship of these conditions is the following. Wolf has shown that (B$'$) implies (A$'$) \cite[ Part 1 Theorem 4.1]{Wolf} and (C$'$) \cite[ Part 2 Theorem 1.1]{Wolf} but nothing else is known. We will show that (C$'$) implies (B$'$) but (A$'$) does not imply (C$'$), and thus (A$'$) does not imply (B$'$). This shows that Wang's conditions restated for $R^p(X)$ have the same relations amongst themselves as Wang's conditions for $R(X)$.

\section{The equivalence of (B$'$) and (C$'$)}

We will first show that (C$'$) implies (B$'$). In particular, we will prove this for arbitrary non-decreasing $\phi$ which allows us to consider the case $t=0$ without loss of generality.

\begin{theorem}
Let $\phi(r)$ be a positive non-decreasing function and let $X$ be a compact subset of the plane and let $x \in X$. Let $2 < p < \infty$ and $q = \frac{p}{p-1} $, and suppose

\begin{equation*}
    \sum_{n=1}^{\infty} 2^{nq}\phi(2^{-n})^{-q}\Gamma_q(A_n(x) \setminus X) < \infty.
\end{equation*}

Then there exists $g \in L^q(X)$ which represents $x$ on $R^p(X)$ such that 

\begin{equation*}
    \int \dfrac{\vert g(z)\vert ^q dm }{\phi(\vert z-x\vert )^q} < \infty.
\end{equation*}

\end{theorem}

\begin{proof}

We may assume that $X$ is a subset of the unit disk and $x=0$ without loss of generality. We abbreviate $A_n(x)$ to $A_n$.  

\bigskip

Let $T_n$ be a linear functional defined by $ T_n(f) =  \int_{\partial A_n} \frac{f(z)}{z}dz $ and let $f$ be a rational function with poles off $X$. We can assume that $f$ is modified off $X$ so that it is continuous, but still analytic in a neighborhood of $X$. Then there exist closed sets $K_n \subseteq A_n \setminus X$ with smooth boundaries such that $f$ is analytic outside of $\cup K_n$ and 

\begin{equation*}
    \int_{A_n \setminus  K_n} \vert f\vert ^p dm \leq 2 \int_{X \cap A_n} \vert f\vert ^p dm.
\end{equation*}

\bigskip We now bound $T_n(f)$. In \cite{Hedberg}, Hedberg constructed smooth functions $P_m$ such that each $P_m=1$ on $ K_m$, has support on $A_{m-1} \cup A_m \cup A_{m+1}$, and $\int \vert \nabla P_m\vert ^q dm \leq C (\Gamma_q(A_m \setminus X) + 4^{-m})$. (Note that in \cite{Hedberg} these functions are called $\phi_m$.) By modifying this construction, we can make it so that $\int \vert \nabla P_m\vert ^q dm \leq C (\Gamma_q(A_m \setminus X) + 4^{-m}\phi(2^{-m})^{q})$. Let $P = \sup_m P_m$ so that $P(z) = 1$ on $\cup K_m$. Then it follows from Green's Theorem and the analyticity of $f$ that

\begin{align*}
    \sum_{m=n}^{\infty} \int_{\partial A_m} \dfrac{f(z)}{z}dz &= \sum_{m=n}^{\infty} \int_{\partial K_m} \dfrac{f(z)}{z}dz\\ 
    &=  \sum_{m=n}^{\infty} \int_{\partial K_m} \dfrac{f(z) P(z)}{z}dz\\
     &=  \sum_{m=n}^{\infty} \int_{A_m \setminus K_m} \dfrac{f(z)}{z} \dfrac{\partial P(z)}{\partial \overline{z}}dm,
\end{align*}

\bigskip \noindent and thus 

\begin{equation*}
    T_n(f) = \sum_{m=n}^{\infty} \int_{\partial A_m} \dfrac{f(z)}{z}dz - \sum_{m=n+1}^{\infty} \int_{\partial A_m} \dfrac{f(z)}{z}dz = \int_{A_n \setminus K_n} \dfrac{f(z)}{z} \dfrac{\partial P(z)}{\partial \overline{z}}dm.
\end{equation*}

\bigskip

\noindent Hence it follows from Holder's inequality that

\begin{align*}
    \vert T_n(f)\vert  &\leq  2^{n+1}  \left(\int_{A_n \setminus K_n} \vert f(z)\vert ^p dm\right)^{\frac{1}{p}}  \left(\int_{A_n \setminus K_n} \vert \nabla P\vert ^q dm\right)^{\frac{1}{q}}\\
    &\leq C 2^n \Vert f\Vert _{L_p(X \cap A_n)} \left [\Gamma_q(A_{n-1} \setminus X) + \Gamma_q(A_n \setminus X) + \Gamma_q(A_{n+1} \setminus X) + 4^{-n}\phi(2^{-n})^q \right]^{\frac{1}{q}}.
\end{align*}

\bigskip \noindent Thus $T_n(f)$ is a bounded linear functional on $R^p(X \cap  A_n  )$. Thus it follows from the Hahn-Banach theorem that there exists $g_n \in L^q(X)$ with support on $ A_n  $ such that $T_n(f) = \int fg_n dm$ for all $f$ in $R^p(X \cap  A_n )$ and 

\begin{equation*} 
\int \vert g_n\vert ^q dm \leq C 2^{nq} \left[ \Gamma_q(A_{n-1} \setminus X) + \Gamma_q(A_n \setminus X) + \Gamma_q(A_{n+1} \setminus X) + 4^{-n}\phi(2^{-n})^q \right].
\end{equation*}

\bigskip

Now let $\displaystyle g = \dfrac{1}{2\pi i}\sum_{n=1}^{\infty} g_n$. Then by the Cauchy integral formula,

\begin{align*}
    f(0)  =&\dfrac{1}{2\pi i} \sum_{n=1}^{\infty} \int_{\partial A_n} \dfrac{f(z)}{z}dz\\
   =&\dfrac{1}{2\pi i} \sum_{n=1}^{\infty} \int f(z) g_n(z) dm\\
   =& \int f(z) \dfrac{1}{2\pi i} \sum_{n=1}^{\infty} g_n(z) dm\\
    =& \int fgdm.
\end{align*}

\bigskip \noindent Thus $g$ represents $0$ on $R^p(X)$. Finally, because $g_n $ has support on $  A_n  $,

\begin{align*}
    \int \dfrac{\vert  g(z)\vert ^q }{\phi(\vert z\vert )^q} dm & = \left( \dfrac{1}{2\pi} \right)^q \int \phi(\vert z\vert )^{-q} \left \vert \sum_{n=1}^{\infty} g_n(z) \right\vert ^qdm\\
   &=\left( \dfrac{1}{2\pi} \right)^q \sum_{m=1}^{\infty} \int_{A_m} \phi(\vert z\vert )^{-q} \left \vert \sum_{n=1}^{\infty} g_n(z) \right\vert ^qdm\\
    &\leq  \sum_{m=1}^{\infty} \int_{A_m}\phi(\vert z\vert )^{-q} \vert g_m(z)\vert ^qdm\\
    &\leq C \sum_{m=1}^{\infty} \phi(2^{-m})^{-q} \int_{A_m}   \vert g_m(z)\vert ^q dm.
\end{align*}

\bigskip \noindent Hence

\begin{equation*}
\begin{aligned}
  \int \dfrac{\vert  g(z)\vert ^q }{\vert \phi(z)\vert ^q} dm &\leq C  \sum_{m=1}^{\infty}  \phi(2^{-m})^{-q}   2^{mq} \left[ \Gamma_q(A_{m-1} \setminus X) + \Gamma_q(A_m \setminus X) + \Gamma_q(A_{m+1} \setminus X) \right. \\ 
  & \left. + 4^{-m}\phi(2^{-m})^q \right]  < \infty.
  \end{aligned}
\end{equation*}
\end{proof}

\section{(A$'$) does not imply (C$'$)}

Finally, we show that (A$'$) does not imply (C$'$). In particular we will show the following.

\begin{theorem}
Let $\phi(r)$ be an admissible function with associated function $\psi(r) = \frac{r}{\phi(r)}$ such that $\phi(0^{+})=0$, and let $2 < p < \infty$ and  $q = \frac{p}{p-1}$. Then there is a compact set $X$ containing $0$ such that for each $\epsilon>0$ the set $\{y\in X: \vert f(y)-f(0)\vert \leq \epsilon \phi(\vert y\vert )\Vert f\Vert _{L^p(X)} \text{ for all } f \in R_0(X)\}$ has full area density at $0$, but

\begin{equation*}
    \sum_{n=1}^{\infty} \phi(2^{-n})^{-q} 2^{nq} \Gamma_q(A_n(0) \setminus X) = \infty.
\end{equation*}

\end{theorem}

\bigskip

Due to the length of the proof, it will be split up into four lemmas. We first describe the construction of the set $X$. Again we abbreviate $A_n(0)$ to $A_n$. Let $a_n = \frac{3}{4}\cdot 2^{-n}$. We can choose a subsequence still denoted $a_n$  such that $\phi(a_{n}) < \frac{1}{2} \phi(a_{n-1})$ and 

\begin{equation*}
\sum_{n=1}^{j-1} \psi(a_n)^{-q} < \psi(a_j)^{-q}.
\end{equation*}

\bigskip

\noindent Let $r_n =[ n^{-1}a_n^q \phi(a_n)^q]^{\frac{1}{2-q}}$. Then there exists $M >0$ such that for $n \geq M$, $\frac{r_n}{a_n}<\frac{1}{5} n^{\frac{-1}{2q}}$. For such $n$ let $D_n$ denote the open disk centered at $a_n$ with radius $r_n$; otherwise, let $D_n$ be the empty set. Let $\Delta$ denote the closed unit disk and let $X = \Delta \setminus \cup_n D_n$.

\begin{lemma}
Let $X$ be the set constructed in the previous paragraph. Then 

\begin{equation*}
    \sum_{n=1}^{\infty} \phi(2^{-n})^{-q} 2^{nq} \Gamma_q(A_n \setminus X) = \infty.
\end{equation*}

\end{lemma}

\begin{proof}

Since the $q$-capacity of a disk of radius $r$ is $r^{2-q}$\cite{Adams}, 

\begin{align*}
    \sum_{n=1}^{\infty} \phi(2^{-n})^{-q} 2^{nq} \Gamma_q(A_n \setminus X) &= \sum_{n=M}^{\infty} \phi(2^{-n})^{-q} 2^{nq} \Gamma_q(A_n \setminus X)\\
    &= \sum_{n=M}^{\infty} \phi(2^{-n})^{-q} 2^{nq} n^{-1} a_{n}^q \phi(a_n)^q\\
    &= \sum_{n=M}^{\infty} \phi(2^{-n})^{-q} 2^{nq} n^{-1} \left(\frac{3}{4}\right)^q  2^{-nq} \phi \left(\frac{3}{4} \cdot 2^{-n} \right)^q\\
    &= \sum_{n=M}^{\infty} \phi(2^{-n})^{-q}  n^{-1} \left(\frac{3}{4}\right)^q  \phi \left(\frac{3}{4} \cdot 2^{-n} \right)^q.
\end{align*}

\bigskip \noindent Then since $\phi(r) = \dfrac{r}{\psi(r)}$ and $\psi(r)$ is non-decreasing

\begin{align*}
     \sum_{n=M}^{\infty} \phi(2^{-n})^{-q}  n^{-1} \left(\frac{3}{4}\right)^q  \phi \left(\frac{3}{4} \cdot 2^{-n} \right)^q
    &= \sum_{n=M}^{\infty} \left(\dfrac{2^{-n}}{\psi(2^{-n})}\right)^{-q} n^{-1} \left(\dfrac{3}{4}\right)^q  \left(\dfrac{\frac{3}{4}\cdot 2^{-n}}{\psi(\frac{3}{4}\cdot 2^{-n})}\right)^q\\
    &= \sum_{n=M}^{\infty}\left( \frac{3}{4}\right)^{2q} \psi(2^{-n})^q n^{-1} \psi \left(\frac{3}{4} \cdot 2^{-n} \right)^{-q}\\
    &\geq \left( \frac{3}{4}\right)^{2q} \sum_{n=M}^{\infty} n^{-1}  = \infty.
\end{align*}

\bigskip 

\noindent Hence \begin{equation*}
    \sum_{n=1}^{\infty} \phi(2^{-n})^{-q} 2^{nq} \Gamma_q(A_n \setminus X) = \infty.
\end{equation*}
\end{proof}

\begin{figure}[ht!]
\begin{center}
\includegraphics[scale = 0.6]{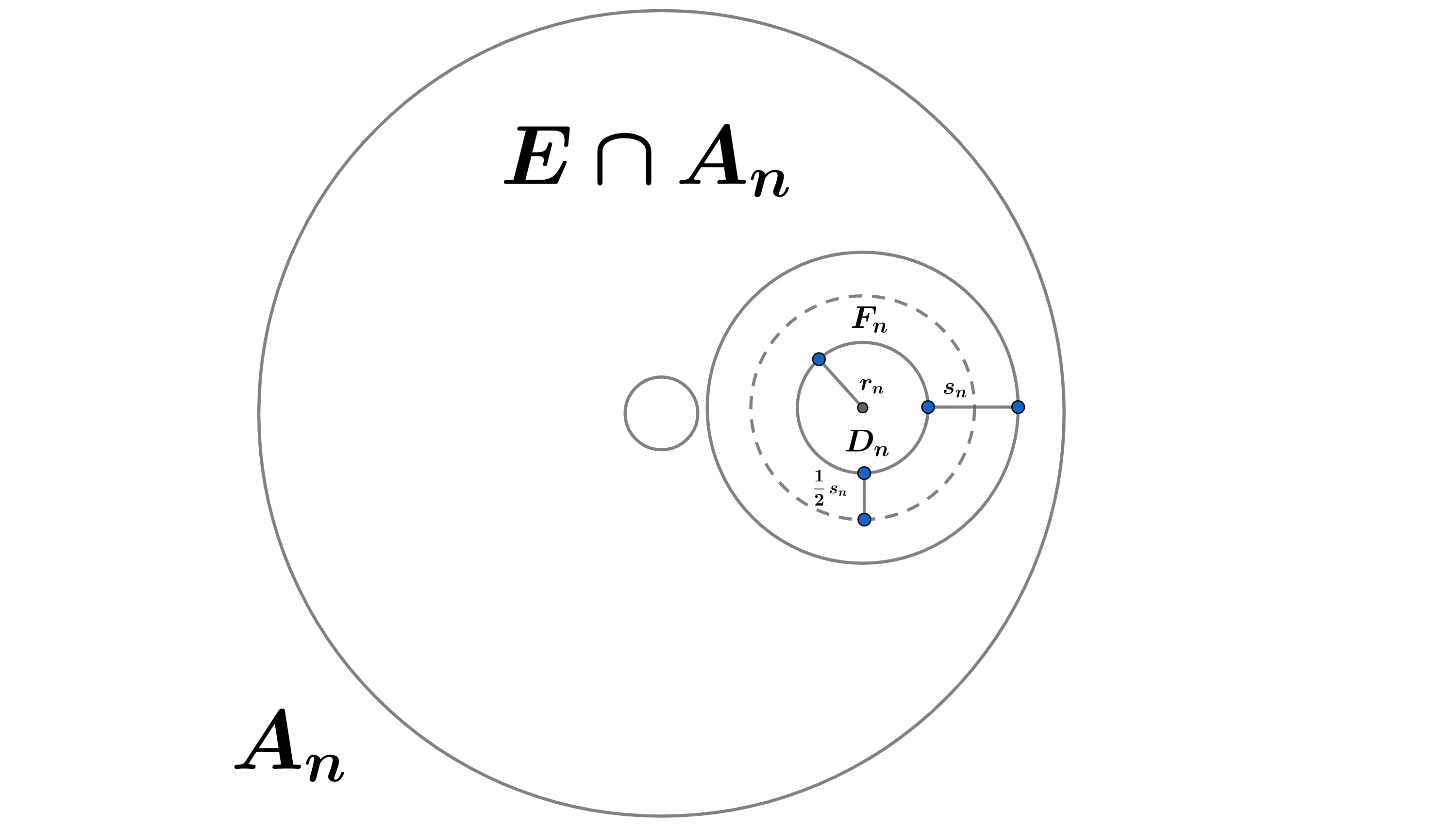}
\caption{The sets $D_n$, $E \cap A_n,$ and $F_n$}
\label{fig1}
\end{center}
\end{figure}

\begin{lemma}
Let $d_n(z)$ denote the distance from $z$ to $D_n$ and let $s_n = \frac{1}{5}n^{-\frac{1}{2q}} a_n$. Let $E = \bigcup_{n>M} \{z \in A_n: d_n(z) \geq r_n\}$. (See Figure 1). Then $E$ has full area density at $0$.

\end{lemma}

\begin{proof}
 Since $\frac{r_n}{a_n}<\frac{1}{5} n^{\frac{-1}{2q}}$, it follows that $r_n < s_n$. Let $B(a,r)$ denote the disk centered at $a$ with radius $r$. Then it follows from the definition of $E$ that $B(0,2^{-j}) \setminus E$ is the union of the disks $B(a_n, r_n+s_n)$, where the union is taken over $n \geq j$. Thus

\begin{equation*}
    \dfrac{m(B(0,2^{-j})) \setminus E)}{m(B(0,2^{-j}))} = \dfrac{\pi \sum_{n=j}^{\infty} (r_n+s_n)^2}{2^{-2j}} = \dfrac{ \sum_{n=j}^{\infty} (r_n+s_n)^2}{ \frac{16}{9} a_j^2}.
\end{equation*}

\bigskip \noindent However, since $r_{n+1} < \frac{1}{2} r_n$ and $s_{n+1} < \frac{1}{2} s_n$,

\begin{align*}
    \dfrac{m(B(0,2^{-j}) \setminus E)}{m(B(0,2^{-j}))} &\leq \dfrac{\left(\sum_{n=0}^{\infty}2^{-2n}\right)(r_j+s_j)^2}{ \frac{16}{9} a_j^2}\\
    &< \frac{3}{4} \left(\dfrac{2s_j}{a_j} \right)^2\\
    &= \dfrac{3}{25} j^{-\frac{1}{q}}.
\end{align*}

\bigskip

\noindent Thus $\dfrac{m(B(0,2^{-j}) \setminus E)}{m(B(0,2^{-j}))} \to 0$ as $j \to \infty$ and $E$ has full area density at $0$.

\end{proof}

\begin{lemma}
\label{lemma3}
There exist smooth functions $P_n$ such that $P_n(z) = 1$ on $D_n$, $\int \vert  \nabla P_n\vert ^q dm \leq C (\Gamma_q(D_n)+16^{-n})$ and $P_n$ is supported on the set $F_n = \{z \in A_n: d_n(z) \leq \frac{s_n}{2}\}$.
\end{lemma}

\begin{proof}

  The proof is a modification of a construction of Hedberg \cite[pg. 277]{Hedberg}. It follows from the definition of $q$-capacity that there exists smooth functions $\omega_n$ such that $\omega_n=1$ on $D_n$ and $\int \vert \nabla \omega_n\vert ^q dm \leq \Gamma_q(D_n) + 16^{-n}$. Let $\xi_n(z)$ be a piecewise linear continuous function of $\vert z\vert $ such that $\xi_n = 1$ on $D_n$, $\xi_n = 0$ outside $F_n$ and $\vert \nabla \xi_n(z)\vert  \leq 2s_n^{-1}$, and let $P_n(z) = \xi_n(z)\omega_n(z)$. Then

\begin{equation*}
    \int \vert \nabla P_n(z)\vert ^q dm \leq  \left( \int \xi_n^q \vert \nabla \omega_n\vert ^q dm + \int \omega_n^q \vert \nabla \xi_n\vert ^q dm \right).
\end{equation*}

\bigskip \noindent The first integral is bounded above by $ \Gamma_q(D_n) + 16^{-n}$. To bound the second integral we first note that it follows from the Gagliardo-Nirenberg-Sobolev inequality (See \cite[pg.277]{Evans}) that $\left(\int \omega_n^{\frac{2q}{2-q}} dm\right)^{\frac{2-q}{2}} \leq C \int \vert \nabla \omega_n\vert ^q dm$. Thus by Holder's inequality

\begin{align*}
    \int \omega_n^q \vert \nabla \xi_n\vert ^q dm &\leq C \left( \int \omega_n^{\frac{2q}{2-q}} dm\right)^{\frac{2-q}{2}} \left( \int_{F_n} \vert \nabla \xi_n\vert ^2 \right)^{\frac{q}{2}}\\
    &\leq C \int \vert \nabla \omega_n\vert ^q dm \cdot m( F_n)^{\frac{q}{2}} s_n^{-q}\\
    &\leq C (\Gamma_q(D_n) + 16^{-n}) \cdot \left (r_n+ \frac{s_n}{2} \right)^q s_n^{-q}\\
    &\leq C (\Gamma_q(D_n) + 16^{-n}) \cdot \left (s_n+ \frac{s_n}{2} \right)^q s_n^{-q}\\
    &= C (\Gamma_q(D_n) + 16^{-n}) \cdot \left (\frac{3}{2} \right)^q.
\end{align*}

\bigskip \noindent Thus $\int \vert \nabla P_n\vert ^q dm \leq C( \Gamma_q(D_n) + 16^{-n})$. 

\end{proof}

\begin{lemma}
\label{lemma4}
Let $\epsilon >0$. Then for $N$ sufficiently large, 

\begin{equation*}
    \vert f(y) - f(0)\vert  \leq \epsilon \phi(\vert y\vert ) \Vert f\Vert _p
\end{equation*}

\bigskip
\noindent whenever $y \in A_N \cap X^{\circ} \cap E$ and $f \in R_0(X)$.

\end{lemma}

\begin{proof}

Let $f \in R_0(X)$ and $y \in A_N \cap X^{\circ} \cap E$. Then it follows from the Cauchy integral formula that 

\begin{align*}
    \vert f(y) -f(0)\vert  &= \left\vert  \frac{1}{2\pi i} \sum_{n=1}^{\infty} \int_{\partial D_n} \dfrac{f(z)dz}{z-y} - \frac{1}{2\pi i} \sum_{n=1}^{\infty} \int_{\partial D_n} \dfrac{f(z)dz}{z} \right\vert \\
    &= \left\vert  \frac{1}{2\pi i} \sum_{n=1}^{\infty} \int_{\partial D_n} \dfrac{y f(z) dz}{z(z-y)} \right\vert \\
 &\leq  \vert y\vert  \left\vert   \sum_{n=1}^{\infty} \int_{\partial D_n} \dfrac{ f(z)  dz}{z(z-y)} \right\vert .
\end{align*}

\bigskip

\noindent Let $P_n$ be the smooth functions constructed in Lemma \ref{lemma3} and let $P = \sup_n P_n$. Then $P(z) = 1$ on $\cup D_n$ and hence by Green's theorem,

\begin{align*}
    \vert f(y) -f(0)\vert  &\leq  \vert y\vert  \left\vert   \sum_{n=1}^{\infty} \int_{\partial D_n} \dfrac{ f(z) P(z) dz}{z(z-y)} \right\vert \\
    &=  \vert y\vert  \left\vert   \sum_{n=1}^{\infty} \int_{F_n \setminus D_n} \dfrac{ f(z) \frac{\partial P}{\partial \overline{z}} dm}{z(z-y)} \right\vert \\
    &\leq   \vert y\vert   \sum_{n=1}^{\infty} \int_{F_n \setminus D_n} \left(a_n-r_n-\frac{s_n}{2} \right)^{-1} \delta_n(y)^{-1} \vert f(z)\vert  (\vert \nabla P(z)\vert ) dm. 
\end{align*}

\bigskip \noindent where $\delta_n(y)$ is the distance from $y$ to $F_n$. Since $\frac{r_n}{a_n} < \frac{1}{5} n^{\frac{-1}{2q}}$,

\begin{equation*}
\left(a_n-r_n-\frac{s_n}{2}\right) \geq \left(1-\frac{1}{5}n^{\frac{-1}{2q}}-\frac{1}{10}n^{\frac{-1}{2q}}\right)a_n \geq \frac{7}{10} a_n,
\end{equation*}

\bigskip \noindent and hence 

\begin{equation*}
    \vert f(y) -f(0)\vert  \leq  C \vert y\vert   \sum_{n=1}^{\infty} \int_{F_n \setminus D_n} a_n^{-1} \delta_n(y)^{-1} \vert f(z)\vert  (\vert \nabla P(z)\vert ) dm.
\end{equation*}

\bigskip

\noindent Next, it follows from Holder's inequality that 

\begin{align*}
    \vert f(y) -f(0)\vert  
    &\leq C\vert y\vert  \left( \sum_{n=1}^{\infty} \int_{F_n \setminus D_n} \vert f(z)\vert ^p dm \right)^{\frac{1}{p}} \left( \sum_{n=1}^{\infty} a_n^{-q} \delta_n(y)^{-q} \int_{F_n \setminus D_n}\vert \nabla P(z)\vert ^q dm\right)^{\frac{1}{q}}\\
    &\leq C\vert y\vert  \cdot \Vert f\Vert _p \left( \sum_{n=1}^{\infty} a_n^{-q} \delta_n(y)^{-q} ( \Gamma_q(D_n) + 16^{-n})\right)^{\frac{1}{q}}\\
    &\leq C\vert y\vert  \cdot \Vert f\Vert _p \left( \sum_{n=1}^{\infty} a_n^{-q} \delta_n(y)^{-q} (r_n^{(2-q)}+16^{-n})\right)^{\frac{1}{q}}.
\end{align*}

\bigskip We next obtain bounds for $\delta_n(y)$. Recall that $\frac{r_n}{a_n} < \frac{1}{5} n^{\frac{-1}{2q}}$ and $y \in A_N$. If $n < N$, then 

\begin{equation*}
    \delta_n(y) \geq \left(a_n - r_n-\frac{s_n}{2}\right) -2^{-N}\geq \frac{7}{10}a_n -2^{-N} = \frac{21}{40} 2^{-n} - 2^{-N} \geq \frac{1}{40} 2^{-n} = \frac{1}{30} a_n.
\end{equation*}

\bigskip \noindent If $n>N$, then

\begin{align*}
    \delta_n(y) &\geq 2^{-(N+1)} - \left(a_n + r_n+ \frac{s_n}{2}\right) \geq 2^{-(N+1)}-\left(1+\frac{3}{10}n^{\frac{-1}{2q}}\right)a_n\\
    &\geq 2^{-(N+1)} - \frac{39}{40} 2^{-n} \geq \frac{1}{40} 2^{-(N+1)} = \frac{1}{60} a_N.
\end{align*}

\bigskip 

\noindent Lastly, if $n=N$, then $\delta_N(y) \geq \frac{s_N}{2}$. Now let $t_n = r_n^{(2-q)}+16^{-n}$. Then

\begin{equation*}
 \begin{aligned}
    \vert f(y) - f(0)\vert  &\leq C\vert y\vert  \cdot \Vert f\Vert _p \left( \sum_{n=1}^{\infty} a_n^{-q} \delta_n(y)^{-q} t_n\right)^{\frac{1}{q}}\\
    &\leq C\vert y\vert  \cdot \Vert f\Vert _p \left( \sum_{n=1}^{N-1} a_n^{-q} \delta_n(y)^{-q} t_n + a_N^{-q} \delta_N(y)^{-q} t_N + \sum_{n=N+1}^{\infty} a_n^{-q} \delta_n(y)^{-q} t_n\right)^{\frac{1}{q}}\\
    &\leq C\vert y\vert  \cdot \Vert f\Vert _p \left( \sum_{n=1}^{N-1} a_n^{-2q}  t_n + a_N^{-q} s_N^{-q} t_N + \sum_{n=N+1}^{\infty} a_n^{-q} a_{N}^{-q} t_n\right)^{\frac{1}{q}}\\
        &\leq C \phi(\vert y\vert ) \cdot \Vert f\Vert _p \left( \sum_{n=1}^{N-1} a_n^{-2q}  t_n \psi(a_N)^{q} + \phi(a_N)^{-q}  s_N^{-q} t_N \right. \\
        &\left. + \sum_{n=N+1}^{\infty} a_n^{-q} \phi(a_N)^{-q}  t_n\right)^{\frac{1}{q}}.
\end{aligned}
\end{equation*}

\bigskip \noindent The last line follows because $\vert y\vert  \leq \frac{4}{3} \phi(\vert y\vert ) \psi(a_N)$. To complete the proof, we must bound the 3 terms inside the parentheses. We begin by bounding the leftmost term

\begin{align*}
   \sum_{n=1}^{N-1} a_n^{-2q}  t_n \psi(a_N)^{q} &= \sum_{n=1}^{N-1} a_n^{-2q}  r_n^{(2-q)} \psi(a_N)^{q}+ \sum_{n=1}^{N-1} a_n^{-2q} 16^{-n} \psi(a_N)^{q}.
\end{align*}

\bigskip 

\noindent We will need to bound both sums. The first sum simplifies thus.

\begin{align*}
    \sum_{n=1}^{N-1} a_n^{-2q}  r_n^{(2-q)} \psi(a_N)^{q}
   &= \psi(a_N)^{q}\sum_{n=1}^{N-1} a_{n}^{-q}  n^{-1} \phi(a_{n})^q\\
    &= \psi(a_N)^{q}\sum_{n=1}^{N-1}   n^{-1} \psi(a_{n})^{-q}.
\end{align*}

\bigskip

\noindent Now we make use of the fact that $\displaystyle \sum_{n=1}^{j-1} \psi(a_n)^{-q} < \psi(a_j)^{-q}$ and $\displaystyle \sum_{n=j}^{N-1} \psi(a_n)^{-q} < \psi(a_N)^{-q}$. Thus 

\begin{align*}
   \psi(a_N)^{q}\sum_{n=1}^{N-1}   n^{-1} \psi(a_{n})^{-q} 
    &= \psi(a_N)^{q} \sum_{n=1}^{j-1}   n^{-1} \psi(a_{n})^{-q} + \psi(a_N)^{q}\sum_{n=j}^{N-1}   n^{-1} \psi(a_{n})^{-q}\\
    &\leq \psi(a_N)^{q} \sum_{n=1}^{j-1}    \psi(a_{n})^{-q} + j^{-1} \psi(a_N)^{q}\sum_{n=j}^{N-1}   \psi(a_{n})^{-q}\\
    &\leq \psi(a_N)^{q} \psi(a_{j})^{-q}+ j^{-1}.
\end{align*}

\bigskip

\noindent To get a bound for the second sum, we note that $a_n^{-2q} = \left(\frac{3}{4}\right)^{-2q} 2^{2qn}$. Hence

\begin{align*}
    \sum_{n=1}^{N-1} a_n^{-2q} 16^{-n} \psi(a_N)^{q}
   &\leq \left(\frac{3}{4}\right)^{-2q}\psi(a_N)^{q}\sum_{n=1}^{\infty} 2^{(2q-4)n}\\
    &= \left(\frac{3}{4}\right)^{-2q}\psi(a_N)^{q}\dfrac{2^{(2q-4)}}{1-2^{(2q-4)}}.
\end{align*}

\bigskip

\noindent Thus we have the following bound for the leftmost term.

\begin{align*}
   \sum_{n=1}^{N-1} a_n^{-2q}  t_n \psi(a_N)^{q} \leq \psi(a_N)^{q} \psi(a_{j})^{-q}+ j^{-1}  + \left(\frac{3}{4}\right)^{-2q}\psi(a_N)^{q}\dfrac{2^{(2q-4)}}{1-2^{(2q-4)}}.
\end{align*}

\bigskip 

\noindent Then by choosing $j$ sufficiently large, it follows that

\begin{equation*}
    \sum_{n=1}^{N-1} a_n^{-2q}  t_n^{(2-q)} \psi(a_N)^{q} \to 0
\end{equation*}

\bigskip

\noindent as $N \to \infty$. We next bound the middle term.

\begin{align*}
    \phi(a_N)^{-q}s_N^{-q}t_N &=\phi(a_N)^{-q}s_N^{-q}(r_N^{2-q}+ 16^{-N})\\
    &= \phi(a_N)^{-q} 5^q N^{\frac{1}{2}} a_N^{-q} (N^{-1} a_N^q \phi(a_N)^q+ 16^{-N})\\
    &= 5^q N^{-\frac{1}{2}} + 5^q \phi(a_N)^{-q} N^{\frac{1}{2}} a_N^{-q} 16^{-N}.
\end{align*}

\bigskip Since $\phi(r) = \dfrac{r}{\psi(r)}$,

\begin{align*}
     5^q\phi(a_N)^{-q} N^{\frac{1}{2}} a_N^{-q} 16^{-N}
    &= 5^q a_N^{-2q} \psi(a_N)^q N^{\frac{1}{2}} 16^{-N} \\
    &= \left(\frac{3}{4\sqrt{5}} \right)^{-2q} \psi(a_N)^q N^{\frac{1}{2}} 2^{(2q-4)N}.
\end{align*}

\bigskip \noindent Thus 

\begin{equation*}
    \phi(a_N)^{-q}s_N^{-q}t_N^{2-q} = 5^q N^{-\frac{1}{2}} + \left(\frac{3}{4\sqrt{5}} \right)^{-2q} \psi(a_N)^q N^{\frac{1}{2}} 2^{(2q-4)N} \to 0
\end{equation*}

\bigskip

\noindent as $N \to \infty$. Finally, we bound the rightmost term.

\begin{align*}
    \sum_{n=N+1}^{\infty} a_n^{-q} \phi(a_N)^{-q}  t_n^{(2-q)} &=
    \sum_{n=N+1}^{\infty} a_n^{-q} \phi(a_N)^{-q}  r_n^{(2-q)} + \sum_{n=N+1}^{\infty} a_n^{-q} \phi(a_N)^{-q}  16^{-n}
\end{align*}

\bigskip

\noindent We will need to bound both sums. To bound the first sum we make use of the property that $\phi(a_n) < \frac{1}{2} \phi(a_{n-1})$.

\begin{align*}
    \sum_{n=N+1}^{\infty} a_n^{-q} \phi(a_N)^{-q}  r_n^{(2-q)}
    &= \phi(a_N)^{-q} \sum_{n=N+1}^{\infty} a_{n}^{-q}   n^{-1} a_{n}^q \phi(a_{n})^q\\
&< \phi(a_N)^{-q} (N+1)^{-1} \sum_{n=N+1}^{\infty}  \phi(a_{n})^q\\
&< \phi(a_N)^{-q} (N+1)^{-1} \sum_{n=1}^{\infty}  \phi(a_N)^q 2^{-nq}\\
&=  (N+1)^{-1} \dfrac{1}{2^q-1}.
\end{align*}

\bigskip

\noindent The second sum of the rightmost term is bounded in a similar way to the second sum of the leftmost term.

\begin{align*}
    \sum_{n=N+1}^{\infty} a_n^{-q} \phi(a_N)^{-q}  16^{-n}
    &=  \sum_{n=N+1}^{\infty} a_n^{-q}a_N^{-q}\psi(a_N)^q 16^{-n}\\
&\leq \psi(a_N)^q \sum_{n=N+1}^{\infty} a_n^{-2q} 16^{-n}\\
&\leq  \left( \frac{3}{4}\right)^{-2q} \psi(a_N)^q \sum_{n=1}^{\infty} 2^{(2q-4)n}\\
&=  \left( \frac{3}{4}\right)^{-2q} \psi(a_N)^q \dfrac{2^{(2q-4)}}{1-2^{(2q-4)}}.
\end{align*}

\bigskip

\noindent Therefore, 

\begin{equation*}
    \sum_{n=N+1}^{\infty} a_n^{-q} \phi(a_N)^{-q}  t_n^{(2-q)} \leq (N+1)^{-1} \dfrac{1}{2^q-1} + \left( \frac{3}{4}\right)^{-2q} \psi(a_N)^q \dfrac{2^{(2q-4)}}{1-2^{(2q-4)}}
\end{equation*}

\bigskip \noindent and hence

\begin{equation*}
    \sum_{n=N+1}^{\infty} a_n^{-q} \phi(a_N)^{-q}  t_n^{(2-q)} \to 0 
\end{equation*}

\bigskip \noindent as $N\to \infty$. Thus all three sums are bounded and tend to $0$ as $N \to \infty$. Hence

\begin{equation*}
    \vert f(y) - f(0)\vert  \leq   \epsilon \phi(\vert y\vert ) \Vert f\Vert _p
\end{equation*}

\bigskip \noindent for $y \in A_N \cap X^{\circ} \cap E$, provided $N$ is chosen sufficiently large.

\end{proof}


\begin{thebibliography}{99}

\bibitem{Adams}
Adams, D. and Hedberg, L.I.:
Function Spaces and Potential Theory.
Springer-Verlag, Berlin, (1996)







\bibitem{Evans}
Evans, L.C.:
Partial Differential Equations.
Graduate Studies in Math., vol. 19, AMS, Providence, RI (1998)


\bibitem{Gamelin}
Gamelin, T.:
Uniform Algebras.
Prentice-Hall (1969)








\bibitem{Hedberg}
Hedberg, L.I.:
Bounded point evaluations and capacity.
J. Funct. Anal. \textbf{10}, 269-280 (1972)









\bibitem{O'Farrell}
O'Farrell, A.G.:
Analytic capacity and equicontinuity
BLMS \textbf{10} 276-279 (1978)




\bibitem{Wang}
Wang, J.:
An approximate Taylor's theorem for $R(X)$.
Math. Scand., \textbf{33}, 343-358

\bibitem{Wang2}
Wang, J.:
Modulus of approximate continuity for $R(X)$.
Math. Scand., \textbf{34}, 219-225 (1974) 

\bibitem{Wolf}
Wolf, E.:
Bounded point evaluations and smoothness properties of functions in $R^p(X)$.
Trans. Amer. Math. Soc. \textbf{238}, 71-88 (1978)


\end{thebibliography}
\end{document}